\documentclass[reqno,12pt,a4paper]{amsart}
\usepackage{amscd,amsfonts,amssymb,amsthm,latexsym}
\usepackage{mathtools}
\usepackage{bm}
\usepackage{srcltx}

\theoremstyle{plain}
\newtheorem{theorem}{Theorem}[section]
\newtheorem*{theorem*}{Theorem}

\newtheorem{lemma}{Lemma}[section]

\theoremstyle{definition}

\newtheorem*{definition*}{Definition}

\theoremstyle{remark}

\newtheorem*{remark*}{Remark}

\numberwithin{equation}{section}

\begin{document}
\raggedbottom %Ќужно, чтобы текст вертикально на раст€гивалс€

\title[Long strings of composite values of polynomials]{Long strings of composite values of polynomials and a basis of order 2}

\author{Artyom Radomskii}

\begin{abstract} We show that for any polynomial $f: \mathbb{Z}\to \mathbb{Z}$ with positive leading coefficient and irreducible over $\mathbb{Q}$, if $N$ is large enough then there are two strings of consecutive positive integers $I_{1}=\{n_1-m,\ldots, n_1+m\}$ and $I_{2}=\{n_2-m, \ldots, n_2+m\}$, such that $m = [(\log N) (\log \log N)^{1/325565}]$, $I_{1}\cup I_{2} \subset [1, N]$, $N = n_1 + n_2$,  and $f(n)$ is composite for any $n\in I_{1}\cup I_{2}$. This extends the result in \cite{Gabdullin.Radomskii} which showed the same result but with $f(n)=n$.
\end{abstract}

 \address{HSE University, Moscow, Russian Federation}


\email{artyom.radomskii@mail.ru}

\maketitle

\section{Introduction}

Set
\[
C(\rho):= \sup\Big\{\delta\in \Big(0,\frac{1}{2}\Big) : \frac{6\cdot 10^{2\delta}}{\log (1/ (2\delta))}<\rho\Big\}.
\]

Our main result is the following

\begin{theorem}\label{T1}
Let $f: \mathbb{Z}\to \mathbb{Z}$ be a polynomial of degree $B\geq 1$ with positive leading coefficient and irreducible over $\mathbb{Q}$. Let $0< \delta < C(1/2)$. Then for every sufficiently large positive integer $N$ there are two strings of consecutive positive integers $I_{1}=\{n_1-m,\ldots, n_1+m\}$ and $I_{2}=\{n_2-m, \ldots, n_2+m\}$, such that $m = [(\log N) (\log \log N)^{\delta}]$, $I_{1}\cup I_{2} \subset [1, N]$, $N = n_1 + n_2$,  and $f(n)$ is composite for any $n\in I_{1}\cup I_{2}$.
\end{theorem}
Theorem \ref{T1} extends a result of Gabdullin and Radomskii \cite{Gabdullin.Radomskii} which showed the same result but with $f(n) = n$. We note that numerical calculations show that $C(1/2)> 1/325565$, and so we can take $\delta = 1/325 565$ in Theorem \ref{T1}. Also, we note that the polynomial $f$ need not have integer coefficients. Indeed, by P\'{o}lya's theorem \cite{Polya}, $f$ is integer valued at integers if and only if $f$ has the form $f(x) = \sum_{j=0}^{B} a_{j} \binom{x}{j}$ with every $a_{j}\in \mathbb{Z}$. In particular, $B! f(x) \in \mathbb{Z}[x]$.

Recall that a set $A \subseteq \mathbb{N}$ is called a basis of order $k$ if every sufficiently large positive integer can be represented as a sum of $k$ summands from $A$. Theorem \ref{T1} implies that the set
\begin{align*}
\{n\geq 3:\,&f(l)\text{ is composite for any }l\in [n-a(n), n+a(n)],\\
 &\text{where } a(n)=[(\log n) (\log\log n)^{\delta}]\}
\end{align*}is a basis of order $2$ for any $\delta < C(1/2)$.

 Let
\[
\Lambda_{f}=\{n\in \mathbb{N}: f(n)\text{ is prime}\}.
\]When $f\in \mathbb{Z}[x]$ is irreducible over $\mathbb{Q}$, has degree two or greater, and the sieving system corresponding to $f$ is non-degenerate, that is
\[
\#\{n\in \mathbb{Z}/p\mathbb{Z}: f(n)\equiv 0\text{ (mod $p$)}\} < p
\]for any prime $p$, it is still an open conjecture (of Bouniakowsky \cite{Bouniakowsky}) that the set $\Lambda_{f}$ is infinite. Moreover it is believed (see the conjecture of Bateman and Horn \cite{Bateman.Horn}, \cite{Bateman.Horn2}) that the number of $n \leq N$ such that $n\in \Lambda_{f}$ should have an asymptotic formula $(\mathfrak{S}(f)+o(1))N/\log N$, where $\mathfrak{S}(f)>0$ is a constant depending only on $f$, and so the gaps of Theorem \ref{T1} would be unusually large compared  to the average gap of size $\asymp_{f} \log N$.

In \cite{FKMPT}, Ford, Konyagin, Maynard, Pomerance, and Tao showed that for any $f$ as in Theorem \ref{T1} and for any $0< \delta < C(1/B)$, for every sufficiently large $N$, there is a string of consecutive positive integers $n\in [1,N]$ of length at least $(\log N) (\log\log N )^{\delta}$ for which $f(n)$ is composite. Ford and Gabdullin \cite{Ford.Gabdullin} improved the result in \cite{FKMPT}, replacing $C(1/B)$ by $C(1)$. We follow the constructions in \cite{FKMPT}, \cite{Ford.Gabdullin}, and \cite{Gabdullin.Radomskii} , modifying some steps in them.

 We also note that the technique proposed by Ford, Konyagin, Maynard, Pomerance, and Tao in \cite{FKMPT} allows to find large gaps between consecutive primes. Theorem 1 in \cite{FKMPT} implies that there is a gap between consecutive primes in $[1, N]$ of size
 \[
 \gg\log N (\log\log N)^{\delta}
 \] for any $\delta < C(1)\approx 1/835$. This is stronger than the trivial bound of $(1+o(1))\log N$, which is immediate from the Prime Number Theorem, but is worse than the current best bounds for this problem. Indeed, the problem of finding large gaps between consecutive primes has a long history, and it is currently known that gaps of size
 \[
 \gg \log N\,\frac{\log\log N\log\log\log\log N}{\log\log\log N}
 \]exist below $N$ if $N$ is large enough, a recent result of Ford, Green, Konyagin, Maynard, and Tao \cite{FGKMT.LARGE}. The key interest is that the technique from \cite{FKMPT} applies to much more general sieving situations, to which it appears difficult to adapt the previous techniques.

\section{Notation}

We use $X\ll Y$, $Y\gg X$, or $X=O(Y)$ to denote the estimate $|X|\leq C Y$ for some constant $C>0$, and write $X\asymp Y$ for $X \ll Y \ll X$. Throughout the remainder of the paper, all implied constants in $O$, $\ll$ or $\gg$ may depend on quantities $B$, $\delta$, $\rho (f)$ \eqref{rho.asymp}, $C_{f}$ \eqref{Density.asymp}, and on the positive parameters $K$, $\xi$, $M$, and $\varepsilon$ which we will describe below. We also assume that the quantity $x$ is sufficiently large in terms of all of these parameters.

The notation $X=o(Y)$ means that $\lim_{x\to \infty} X/Y = 0$, and the notation $X\sim Y$  means that $\lim_{x\to \infty} X/Y = 1$.

If $S$ is a statement, we use $1_{S}$ to denote its indicator, thus $1_{S}=1$ when $S$ is true and $1_{S}=0$ when $S$ is false.

We will rely on probabilistic methods in this paper. Boldface symbols such as $\mathbf{b}$, $\mathbf{n}$, $\mathbf{S}$, $\bm{\lambda}$, etc. denote random variables (which may be real numbers, random sets, random functions, etc.) Most of these random variables will be discrete (in fact they will only take on finitely many values), so that we may ignore any technical issues of measurability. We use $\mathbb{P}(\mathbf{E})$ to denote the probability of a random event $\mathbf{E}$, and $\mathbb{E}(\mathbf{X})$ to denote the expectation of the random (real-valued) variable $\mathbf{X}$.

The symbol $p$ (as well as variants such as $p_1$, $p_2$, etc.) will always denote a prime.

If $x$ is a real number, then $[x]$ denotes its integral part.

By $\# A$ we denote the number of elements of a finite set $A$.

\section{Setup}\label{SECTION.SETUP}

We first note that $\delta \in (0,1/2)$ satisfies $\delta< C(\rho)$ if and only if
\begin{equation}\label{C.RHO.INEQUALITY}
\frac{6\cdot 10^{2\delta}}{\log (1/ (2\delta))}<\rho.
\end{equation}

For a fixed $\delta\in (10^{-6}, C(1/2))$, we define
\begin{equation}\label{SETUP:Def.y}
y= [x(\log x)^{\delta}]\qquad\text{and}\qquad z=\frac{y\log\log x}{(\log x)^{1/2}}.
\end{equation} We note that $\log y \sim \log x$ and $\log z \sim \log x$.

Let $\widetilde{f}(x) = B! f(x)$. Hence, $\widetilde{f}\in \mathbb{Z}[x]$. We set $I_{p} = \emptyset$ if $p \leq B$ or $p$ divides the leading coefficient of $f$, and
\[
I_p := \{n\in \mathbb{Z}/p\mathbb{Z}: \widetilde{f}(n)\equiv 0\text{ (mod $p$)}\}
\]otherwise. By Lagrange's theorem, we have
\[
\#I_p < p\text{\quad and \quad}\#I_p \leq B
\] for any $p$. Since $f$ is irreducible, from Landau's Prime Ideal Theorem \cite{Landau}
(see also \cite[pp. 35-36]{Cojocaru.Murty}) we have
\begin{equation}\label{SETUP:PRIME.IDEAL.TH}
\sum_{p\leq x} \#I_{p} = \textup{li}\,x + O(x\,\textup{exp}(-c\sqrt{\log x}))
\end{equation} for some $c>0$, and so by partial summation we see that
\begin{equation}\label{Ip.asymp}
\sum_{p\leq x} \frac{\#I_p}{p} = \ln\ln x + c_{f} + o(1),
\end{equation}where $c_f$ is a constant depending only on $f$. By $q$ we always denote a prime with $\#I_{q} \geq 1$. In particular, this implies that $q>B$.

 By Chebotarev Density Theorem \cite{Chebotarev} (see also \cite{Lagarias.Odlyzko}) there exist $\rho = \rho(f)> 0$ such that
\begin{equation}\label{rho.asymp}
\lim_{x\to \infty} \frac{\#\{p\leq x: \#I_{p}\geq1\}}{x/\log x} = \rho,
\end{equation} and for each $\nu\in \{1,\ldots, B\} $ there exists $\rho_{\nu} = \rho_{\nu} (f)\geq 0$ such that
 \begin{equation}\label{SETUP:CHEBOTAREV.NU}
 \lim_{x\to \infty} \frac{\#\{p\leq x: \#I_{p}=\nu\}}{x/\log x} = \rho_{\nu}.
 \end{equation}We set
 \[
 \mathcal{N} = \{1\leq \nu \leq B: \rho_{\nu}>0\}.
 \]

We see from \eqref{SETUP:PRIME.IDEAL.TH} and \eqref{SETUP:CHEBOTAREV.NU} that
\begin{equation}\label{SUM.nu.rho}
\sum_{\nu \in \mathcal{N}} \nu \rho_{\nu} = 1,
\end{equation} and from \eqref{rho.asymp} and \eqref{SETUP:CHEBOTAREV.NU} that
\[
\sum_{\nu \in \mathcal{N}} \rho_{\nu} = \rho.
\]

We define
\begin{gather*}
P(x):= \prod_{q\leq x}q,\quad\qquad \sigma (x):= \prod_{q\leq x}\bigg(1-\frac{\#I_{q}}{q}\bigg),\\
 P (z,x):=\prod_{z<q\leq x}q,\quad\qquad \sigma (z,x):= \prod_{z<q\leq x}\bigg(1-\frac{\#I_{q}}{q}\bigg).
\end{gather*} By the Prime Number Theorem, we have
\begin{equation}\label{SETUP:P.x.EST}
P (x)\leq \prod_{p\leq x} p =\textup{exp}((1+o(1))x),
\end{equation}and from \eqref{Ip.asymp} we obtain
\begin{equation}\label{Density.asymp}
\sigma (x) \sim \frac{C_{f}}{\log x},
\end{equation} where $C_{f}>0$ is a constant depending only on $f$.

 For any integer $b$, we define
\begin{align*}
S_{x}(b)&:= \{n\in \mathbb{Z}:\ n-b \not \equiv \alpha\text{ (mod $q$) for any $\alpha\in I_{q}$ for any $q\leq x$}\},\\
S_{z,x}(b)&:= \{n\in \mathbb{Z}:\ n-b \not \equiv \alpha\text{ (mod $q$) for any $\alpha \in I_{q}$ for any $z<q\leq x$}\}.
\end{align*} We denote $S_x:= S_{x}(0)$. It is clear that $S_{x}(b)= S_{x}+ b$. Here $S_{x}+b:= \{s+b:\ s\in S_{x}\}$.

Let $\xi > 1$. We set
\[
\mathfrak{H}:=\bigg\{H\in\{1, \xi, \xi^{2},\ldots\}:\ \frac{2y}{x}\leq H\leq \frac{y}{\xi z}\bigg\}
\]and
\[
\mathfrak{H}'=\{H\in \mathfrak{H}: H=\xi^{j},\ j\text{ is even}\},\qquad \mathfrak{H}''=\{H\in \mathfrak{H}: H=\xi^{j},\ j\text{ is odd}\}.
\]It is clear that
\[
\bigsqcup_{H\in\mathfrak{H}}\Big(\frac{y}{\xi H}, \frac{y}{H}\Big]\subset \Big(z, \frac{x}{2}\Big],
\]and
\begin{equation}\label{H_range}
(\log x)^{\delta}\leq H\leq \frac{(\log x)^{1/2}}{\log\log x}\quad \qquad (H\in \mathfrak{H}).
\end{equation} For each $H\in \mathfrak{H}$ and $\nu \in \mathcal{N}$, let
 \[
 \mathcal{Q}_{H,\nu}=\Big\{q\in \Big(\frac{y}{\xi H}, \frac{y}{H}\Big]: \#I_{q} = \nu\Big\}.
 \]It follows from \eqref{SETUP:CHEBOTAREV.NU} that
\begin{equation}\label{QH.NU.ASYMPT}
\# \mathcal{Q}_{H, \nu}\sim \rho_{\nu}(1-1/\xi)\frac{y}{H\log x}.
\end{equation}Also, let
\[
\mathcal{Q}'_{\nu} = \bigcup_{H\in \mathfrak{H}'}\mathcal{Q}_{H, \nu},\qquad \mathcal{Q}' = \bigcup_{\nu \in \mathcal{N}} \mathcal{Q}'_{\nu},
\]and
\[
\mathcal{Q}''_{\nu} = \bigcup_{H\in \mathfrak{H}''}\mathcal{Q}_{H, \nu},\qquad \mathcal{Q}'' = \bigcup_{\nu \in \mathcal{N}} \mathcal{Q}''_{\nu}.
\]Let
\[
\mathcal{Q}=\mathcal{Q}'\cup \mathcal{Q}''.
\]

We note that $\mathcal{Q}'\cap \mathcal{Q}''=\emptyset$, and if $q\in \mathcal{Q}$, then $q\in (z, x/2]$. It is clear that
\begin{equation}\label{SETUP:Q.EST}
\#\mathcal{Q} < \#\Big\{p: p\leq \frac{x}{2}\Big\}\leq \frac{x}{\log x}   < y.
\end{equation} For $q\in \mathcal{Q}$, let $\nu_{q}:= \#I_{q}$ and let $H_{q}$ be the unique element of $\mathfrak{H}$ such that
\[
\frac{y}{\xi H_{q}}< q\leq \frac{y}{H_{q}}.
\]

We denote by $\mathbf{b}$ a random residue class from $\mathbb{Z}/P\mathbb{Z}$, chosen with uniform probability, where we adopt the abbreviations
\[
P=P (z),\qquad \sigma=\sigma (z),\qquad \mathbf{S}'= S_{z}(\mathbf{b}),\qquad \mathbf{S}''= S_{z}(-N -\mathbf{b}).
\]Let
\[
6<M< 7.
\]For $H\in\mathfrak{H}$, we use the notation
\[
P_{1}=P (H^{M}),\qquad \sigma_{1}=\sigma (H^{M}),\qquad P_{2}=P (H^{M}, z),\qquad \sigma_{2}=\sigma (H^{M}, z),
\]and define
\begin{gather*}
\mathbf{b}_{1}\equiv \mathbf{b}\ \text{(mod $P_{1}$)},\qquad
\mathbf{S}'_{1}= S_{H^{M}}(\mathbf{b}_{1}),\qquad \mathbf{S}''_{1}= S_{H^{M}}(-N-\mathbf{b}_{1}),\\
\mathbf{b}_{2}\equiv \mathbf{b}\ \text{(mod $P_{2}$)},\qquad \mathbf{S}'_{2}= S_{H^{M}, z}(\mathbf{b}_{2}),
\qquad \mathbf{S}''_{2}= S_{H^{M}, z}(-N -\mathbf{b}_{2}),
\end{gather*}
with the convention that $\mathbf{b}_{1}\in \mathbb{Z}/P_{1}\mathbb{Z}$ and $\mathbf{b}_{2}\in \mathbb{Z}/P_{2}\mathbb{Z}$. Thus, $\mathbf{b_1}$ and $\mathbf{b_{2}}$ are each uniformly distributed, are independent of each other, and likewise $\mathbf{S}'_{1}$ and $\mathbf{S}'_{2}$ are independent, and $\mathbf{S}''_{1}$ and $\mathbf{S}''_{2}$ are independent. We also have the obvious relations
\[
P=P_1 P_2,\qquad \sigma=\sigma_1\sigma_2,\qquad \mathbf{S}'=\mathbf{S}'_{1}\cap \mathbf{S}'_{2},\qquad \mathbf{S}''=\mathbf{S}''_{1}\cap \mathbf{S}''_{2}.
\] From \eqref{Density.asymp} we have
\begin{equation}\label{SIGMA.2}
\sigma_{2}^{-1}=\prod_{H^{M}< q \leq z}\bigg(1-\frac{\#I_{q}}{q}\bigg)^{-1}\sim \frac{\log z}{M\log H}\sim \frac{\log x}{M\log H}.
\end{equation} For $q\in \mathcal{Q}$, let
\[
I_{q}=\{\alpha_{q, i}: 1\leq i \leq \nu_{q}\},
\]where $\alpha_{q,i}\in [1,q]$ for all $i\leq \nu_{q}$. For $n\in \mathbb{Z}$, we set
\begin{align*}
\mathbf{AP}'(J; q, n)&:=\bigg(\bigsqcup_{i=1}^{\nu_{q}}\{n+\alpha_{q,i}+qh:\ 1\leq h \leq J\}\bigg)\cap \mathbf{S}'_{1},\\
\mathbf{AP}''(J; q, n)&:=\bigg(\bigsqcup_{i=1}^{\nu_{q}}\{n+\alpha_{q,i}-qh:\ 1\leq h \leq J\}\bigg)\cap \mathbf{S}''_{1}.
\end{align*} Let
\begin{equation}\label{Def_lambda}
\bm{\lambda}'(H; q, n)= \frac{1_{\mathbf{AP}'(KH; q, n)\subset \mathbf{S}'_{2}}}{\sigma_{2}^{\# \mathbf{AP}'(KH; q, n)}}\qquad\text{and}\qquad    \bm{\lambda}''(H; q, n)= \frac{1_{\mathbf{AP}''(KH; q, n)\subset \mathbf{S}''_{2}}}{\sigma_{2}^{\# \mathbf{AP}''(KH; q, n)}}.
\end{equation}

Suppose that $n\in [1,y]$, $q\in \mathcal{Q}'$, $i \leq \nu_{q}$, and $1 \leq h \leq K H_{q}$. Then $n- \alpha_{q, i} - qh \leq y$ and
\[
n- \alpha_{q, i} - qh \geq 1 - q - q K H_{q}\geq 1- q - Ky > - (K+1)y.
\]Hence,
\begin{equation}\label{n.qh.range.1}
- (K+1)y < n- \alpha_{q, i} - qh\leq y.
\end{equation}Similarly, let $n\in [-y,-1]$, $q\in \mathcal{Q}''$, $i \leq \nu_{q}$, and $1 \leq h \leq K H_{q}$. Then
\[
n - \alpha_{q,i}+ qh \geq -y- q +q = - y,
\]and
\[
n - \alpha_{q,i}+ qh \leq - 1 - 1 + q K H_{q} \leq - 2 + Ky < (K+1)y.
\]Thus, we obtain
\begin{equation}\label{n.qh.range.2}
-y \leq n - \alpha_{q,i}+ qh < (K+1)y.
\end{equation}

\section{Main Lemmas}

\begin{lemma}\label{L1}
Let $10^{-6}< \delta < C(1/2)$. Then there exist $6<M < 7$, $\xi>1$, $K>0$, and $0 < \varepsilon < (M-6)/7$ such that  for sufficiently large $x$ (with respect to $\delta$, $M$, $\xi$, $K$, and $\varepsilon$) there exist an integer $b$ \textup{mod $P(z)$} (a choice of $\mathbf{b}$) and non-empty sets $\mathcal{R}'_{\nu}\subseteq \mathcal{Q}'_{\nu}$ and $\mathcal{R}''_{\nu}\subseteq \mathcal{Q}''_{\nu}$ for $\nu \in \mathcal{N}$ such that the following statements hold.

\textup{(i)} One has
\begin{equation}\label{S.EST}
\#(S'\cap [1, y])\leq 2\sigma y\qquad\text{and}\qquad \#(S''\cap [-y, -1])\leq 2\sigma y.
\end{equation}

\textup{(ii)} For all $q\in \mathcal{R}'= \bigcup_{\nu \in \mathcal{N}}\mathcal{R}'_{\nu}$ one has
\begin{equation}\label{L1:lambda.1}
\sum_{-(K+1)y< n \leq y} \lambda ' (H_{q}; q, n) = \left(1+O\left(\frac{1}{(\log x)^{\delta (1+\varepsilon)}}\right)\right)(K+2)y.
\end{equation}

For all $q\in \mathcal{R}''= \bigcup_{\nu \in \mathcal{N}}\mathcal{R}''_{\nu}$ one has
\begin{equation}\label{L1:lambda.2.N}
\sum_{-y\leq n < (K+1) y} \lambda ''(H_{q}; q, n) = \left(1+O\left(\frac{1}{(\log x)^{\delta (1+\varepsilon)}}\right)\right)(K+2)y.
\end{equation}

\textup{(iii)} For each $\nu\in \mathcal{N}$ and $i\in\{1,\ldots, \nu\}$, there exists a set $V'_{\nu, i}\subset (S '\cap [1,y])$ such that
\begin{equation}\label{V.1.DOPOLNENIYE}
\# ((S'\cap [1, y]) \setminus V'_{\nu, i}) \leq \frac{\rho x}{10 B^{2}\log x}
\end{equation}and for any $n\in V'_{\nu, i}$, one has
\begin{equation}\label{L1:lambda.1.struya}
\sum_{q\in \mathcal{R}'_{\nu}} \sum_{h\leq K H_{q}}\lambda ' (H_{q}; q, n-\alpha_{q,i}-qh)=\bigg(C'_{\nu} + O\bigg(\frac{1}{(\log x)^{\delta (1+\varepsilon)}}\bigg)\bigg) (K+2)y
\end{equation}for some quantity $C'_{\nu}$ independent of $n$ and $i$ with
\begin{equation}\label{C1.nu.INEQ.BASE}
 10^{2\delta} \rho_{\nu} \leq  C'_{\nu} \leq 100 \rho_{\nu}.
\end{equation}
\textup{(iv)} For each $\nu\in \mathcal{N}$ and $i\in\{1,\ldots, \nu\}$, there exists a set $V''_{\nu, i}\subset (S ''\cap [-y,-1])$ such that
\begin{equation}\label{V.2.DOPOLNENIYE}
\# ((S''\cap [-y, -1]) \setminus V''_{\nu, i}) \leq \frac{\rho x}{10 B^{2}\log x}
\end{equation}and for any $n\in V''_{\nu, i}$, one has
\begin{equation}\label{L1:lambda.2.struya}
\sum_{q\in \mathcal{R}''_{\nu}} \sum_{h\leq K H_{q}}\lambda '' (H_{q}; q, n-\alpha_{q,i}+qh)=\bigg(C''_{\nu} + O\bigg(\frac{1}{(\log x)^{\delta (1+\varepsilon)}}\bigg)\bigg) (K+2)y
\end{equation}for some quantity $C''_{\nu}$ independent of $n$ and $i$ with
\[
 10^{2\delta} \rho_{\nu} \leq  C''_{\nu} \leq 100 \rho_{\nu}.
\]
\end{lemma}

\begin{lemma}\label{L2}
Suppose that $0<\delta \leq 1/2$ and $K_{0}\geq 1$, and let $y\geq y_{0}(\delta, K_{0})$ with $y_{0}(\delta, K_{0})$ sufficiently large, and let $V$ be a finite set with $\#V\leq y$. Let $1\leq s \leq y$, and suppose that $\mathbf{e}_{1},\ldots, \mathbf{e}_{s}$ are random subsets of $V$ satisfying the following:
\begin{align}
\#\mathbf{e}_{i}&\leq \frac{K_{0}(\log y)^{1/2}}{\log\log y}\qquad (1\leq i\leq s),\label{L3:I}\\
\mathbb{P}(v\in \mathbf{e}_{i})&\leq y^{-1/2 - 1/100}\qquad (v\in V,\ 1\leq i \leq s),\label{L3:II}\\
\sum_{i=1}^{s}\mathbb{P}(v, v' \in \mathbf{e}_{i})&\leq y^{-1/2}\qquad (v, v' \in V,\ v\neq v'),\label{L3:III}\\
\bigg|\sum_{i=1}^{s} \mathbb{P}(v\in \mathbf{e}_{i}) - C_1 \bigg|&\leq \eta\qquad (v\in V),\label{L3:IV}
\end{align}where $C_1$ and $\eta$ satisfy
\[
10^{2\delta}\leq C_1\leq 100,\qquad \eta \geq \frac{1}{(\log y)^{\delta}\log\log y}.
\]Then there are subsets $e_i$ of $V$, $1\leq i \leq s$, with $e_i$ being in the support of $\mathbf{e}_{i}$ for every $i$, and such that
\[
\# \bigg(V\setminus \bigcup_{i=1}^{s} e_i\bigg)\leq C_{0}\eta \#V,
\]where $C_0 >0$ is an absolute constant.
\end{lemma}
\begin{proof}
This is \cite[Lemma 3.1]{FKMPT}.
\end{proof}

\section{Deduction of Theorem \ref{T1} from Lemmas \ref{L1} and \ref{L2}}\label{SECTION.DEDUCE}

Let all hypotheses of Lemma \ref{L1} hold. We put
\[
V' = \bigcap_{\nu\in \mathcal{N}}\bigcap_{1\leq i\leq \nu} V'_{\nu, i}.
\]From part \textup{(iii)} of Lemma \ref{L1}, we obtain $V' \subset (S' \cap [1, y])$,
\begin{align*}
\#((S' \cap [1, y])\setminus V') &\leq \sum_{\nu \in \mathcal{N}}\sum_{1\leq i \leq \nu}\#((S' \cap [1, y])\setminus V'_{\nu, i})\\
&\leq \sum_{\nu \in \mathcal{N}}\sum_{1\leq i \leq \nu} \frac{\rho x}{10 B^{2} \log x}\leq \frac{\rho x}{10 \log x},
\end{align*}and if $n\in V'$, then for each $\nu\in \mathcal{N}$ and $i\in\{1,\ldots,\nu\}$ the statement \eqref{L1:lambda.1.struya} holds.

For each $q\in \mathcal{R}'$, we choose a random integer $\mathbf{n}'_{q}$ with probability density function
\begin{equation}\label{Def_n}
\mathbb{P}(\mathbf{n}'_{q}=n)=\frac{\lambda '(H_{q}; q, n)}{\sum_{-(K+1)y<t\leq y}\lambda '(H_{q}; q, t)}\qquad
(-(K+1)y< n \leq y).
\end{equation} Note that by \eqref{L1:lambda.1} the denominator is non-zero, so that this is a well-defined probability distribution. For $q\in \mathcal{R}'$, we define the random set
\[
\mathbf{e}'_{q} = \bigsqcup_{1 \leq i \leq \nu_{q}} \mathbf{e}'_{q, i},
\]where
\[
\mathbf{e}'_{q, i} = V' \cap \{\mathbf{n}'_{q}+ \alpha_{q,i} + qh: 1\leq h\leq KH_{q}\},\quad i=1,\ldots, \nu_{q}.
\]In particular, we obtain that $\mathbf{e}'_{q} \subset V'$ for any $q\in \mathcal{R}'$.

We are going to apply Lemma \ref{L2} with $V=V'$, $\{\mathbf{e}_{1},\ldots, \mathbf{e}_{s}\}=\{\mathbf{e}'_{q}:\ q\in \mathcal{R}'\}$, $s=\# \mathcal{R}'$, $K_{0}=BK$, and
\begin{equation}\label{ETA.DEF}
\eta=\frac{1}{(\log y)^{\delta}\log \log y}.
\end{equation} We note that $s\geq 1$, since $\mathcal{R}'\neq \emptyset$, and by \eqref{SETUP:Q.EST} we have $s\leq \#\mathcal{Q}\leq y.$

If $q\in \mathcal{R}'$, then from \eqref{H_range} we have
\[
\#\mathbf{e}'_{q}\leq BKH_{q}\leq \frac{BK(\log x)^{1/2}}{\log \log x}\leq   \frac{K_{0}(\log y)^{1/2}}{\log \log y}.
\] Hence, \eqref{L3:I} holds.

For $n\in V'$ and $q\in \mathcal{R}'$, we have from \eqref{n.qh.range.1}, \eqref{Def_n}, \eqref{L1:lambda.1}, and \eqref{Def_lambda} that
\begin{align}
\mathbb{P}(n\in \mathbf{e}'_{q})&=\sum_{1\leq i \leq \nu_{q}}\sum_{1\leq h \leq KH_{q}} \mathbb{P}(\mathbf{n}'_{q}=n -\alpha_{q,i}-qh)\notag\\
&=\sum_{1\leq i \leq \nu_{q}}\sum_{1\leq h \leq K H_{q}} \frac{\lambda '(H_{q}; q, n-\alpha_{q,i}-qh)}{\sum_{-(K+1)y <t\leq y}\lambda '(H_{q}; q, t)}\notag\\
&\ll \frac{1}{y} H_{q} \sigma_{2}^{-BK H_q}\ll y^{-9/10},\label{PROB.n.Eq.1}
\end{align} which gives \eqref{L3:II}.

For $n\in V'$, we have from \eqref{Def_n}, \eqref{L1:lambda.1.struya}, and \eqref{L1:lambda.1} that
\begin{align*}
\sum_{q\in \mathcal{R}'} \mathbb{P}(n\in \mathbf{e}'_{q})&=
\sum_{\nu\in \mathcal{N}}\sum_{q\in \mathcal{R}'_{\nu}}\mathbb{P}(n\in \mathbf{e}'_{q})=
\sum_{\nu\in \mathcal{N}}\sum_{q\in \mathcal{R}'_{\nu}}\sum_{i \leq \nu} \sum_{h \leq KH_{q}} \mathbb{P}(\mathbf{n}'_{q}=n-\alpha_{q,i}-qh)\\
&=\sum_{\nu\in \mathcal{N}}\sum_{i \leq \nu}\sum_{q\in \mathcal{R}'_{\nu}} \sum_{h \leq KH_{q}}
\frac{\lambda '(H_{q}; q, n-\alpha_{q,i}-qh)}{\sum_{-(K+1)y<t\leq y}\lambda '(H_{q}; q, t)}\\
&= C' + O\big((\log x)^{-\delta(1+ \varepsilon)}\big),
\end{align*}where
\[
C' = \sum_{\nu\in \mathcal{N}}\nu C'_{\nu}.
\]From \eqref{SUM.nu.rho} and \eqref{C1.nu.INEQ.BASE} we obtain $10^{2\delta}\leq C' \leq 100$. Thus, \eqref{L3:IV} follows.

We now turn to \eqref{L3:III}. In \cite{Ford.Gabdullin} it was shown that for any positive integer $m$
\begin{equation}\label{Q.N.Iq.EST}
\#\{q: m\textup{ (mod $q$)}  \in I_{q} - I_{q}\}\ll \log (m+1)
\end{equation} (here, as usual, $A-B :=\{a-b: a\in A, b\in B\}$). For any distinct $n, m \in V'$,  we have
\begin{align}\label{PROB.n.m.in.Eq}
\sum_{q\in \mathcal{R}'}& \mathbb{P}(n, m \in \mathbf{e}'_{q})\leq \sum_{q\in \mathcal{R}'}
\bigg(\sum_{1\leq i \leq \nu_{q}} \mathbb{P}(n, m \in \mathbf{e}'_{q,i})+ \sum_{\substack{1 \leq i, j\leq \nu_{q}\\ i\neq j}}\mathbb{P}(n\in \mathbf{e}'_{q,i}, m\in \mathbf{e}'_{q,j})\bigg)\notag\\
&=\sum_{1 \leq i \leq B}\sum_{\substack{q\in \mathcal{R}':\\ i \leq \nu_{q}}} \mathbb{P}(n, m \in \mathbf{e}'_{q,i})
+ \sum_{\substack{1 \leq i, j\leq B\\ i\neq j}}\sum_{\substack{q\in \mathcal{R}':\\ i, j \leq \nu_{q}}}\mathbb{P}(n\in \mathbf{e}'_{q,i}, m\in \mathbf{e}'_{q,j}).
\end{align}If both $n$, $m$ belong to some $\mathbf{e}'_{q,i}$, then $q$ divides $n-m$. But $0<|n-m|<y$ and $q>z> y^{3/4}$, hence there is at most one such $q$. Further, if $n\in \mathbf{e}'_{q,i}$ and $m\in \mathbf{e}'_{q, j}$ for some $q$ and $i\neq j$, then $n-m \equiv \alpha_{q,i} - \alpha_{q,j}$ (mod $q$) and hence $n-m$ (mod $q$) $\in I_{q} - I_{q}$. By \eqref{Q.N.Iq.EST} and the bound $|n-m|<y$, the number of such $q$ is $\ll \log y$. Thus, by \eqref{PROB.n.Eq.1} and \eqref{PROB.n.m.in.Eq},
\[
\sum_{q\in \mathcal{R}'} \mathbb{P}(n, m \in \mathbf{e}'_{q}) \ll y^{-0.9} \log y < y^{-0.5},
\] which gives \eqref{L3:III}.

Thus, all hypotheses of Lemma \ref{L2} hold, and for each $q\in \mathcal{R}'$ there is a number $n'_{q}$ such that if we put
\[
e'_{q, i} = V' \cap \{n'_{q}+ \alpha_{q,i} + qh: 1\leq h\leq KH_{q}\},\quad i=1,\ldots, \nu_{q},
\]and
\[
e'_{q} = \bigsqcup_{1 \leq i \leq \nu_{q}} e'_{q, i},
\]then we have
\[
\# \Big(V'\setminus \bigcup_{q\in \mathcal{R}'} e'_{q}\Big)\leq C_{0}\eta\#V'.
\]By \eqref{S.EST}, we have
\[
\#V'\leq 2 \sigma y \ll \frac{x (\log x)^{\delta}}{\log x}.
\] Thus, by \eqref{ETA.DEF},
\[
\# \Big(V'\setminus \bigcup_{q\in \mathcal{R}'} e'_{q}\Big) \ll \frac{x}{\log x \log\log x}
< \frac{\rho x}{10 \log x}.
\]

Similarly, we put
\[
V'' = \bigcap_{\nu\in \mathcal{N}}\bigcap_{1\leq i\leq \nu} V''_{\nu, i}.
\]From part \textup{(iv)} of Lemma \ref{L1}, we obtain $V'' \subset (S'' \cap [-y, -1])$,
\[
\#((S'' \cap [-y, -1])\setminus V'') \leq  \frac{\rho x}{10 \log x},
\]and if $n\in V''$, then for each $\nu\in \mathcal{N}$ and $i\in\{1,\ldots,\nu\}$ the statement \eqref{L1:lambda.2.struya} holds.

For each $q\in \mathcal{R}''$, we choose a random integer $\mathbf{n}''_{q}$ with probability density function
\[
\mathbb{P}(\mathbf{n}''_{q}=n)=\frac{\lambda '(H_{q}; q, n)}{\sum_{-y\leq t< (K+1)y}\lambda ''(H_{q}; q, t)}\qquad
(-y\leq n < (K+1)y).
\] By \eqref{L1:lambda.2.N} the denominator is non-zero, so that this is a well-defined probability distribution. For $q\in \mathcal{R}''$, we define the random set
\[
\mathbf{e}''_{q} = \bigsqcup_{1 \leq i \leq \nu_{q}} \mathbf{e}''_{q, i},
\]where
\[
\mathbf{e}''_{q, i} = V'' \cap \{\mathbf{n}''_{q}+ \alpha_{q,i} - qh: 1\leq h\leq KH_{q}\},\quad i=1,\ldots, \nu_{q}.
\]In particular, we obtain that $\mathbf{e}''_{q} \subset V''$ for any $q\in \mathcal{R}''$.

We are going to apply Lemma \ref{L2} with $V=V''$, $\{\mathbf{e}_{1},\ldots, \mathbf{e}_{s}\}=\{\mathbf{e}''_{q}:\ q\in \mathcal{R}''\}$, $s=\# \mathcal{R}''$, $K_{0}=BK$, and $\eta$ given by \eqref{ETA.DEF}.
Arguing as above, we see that all hypotheses of Lemma \ref{L2} hold, and hence for each $q\in \mathcal{R}''$ there is a number $n''_{q}$ such that if we set
\[
e''_{q, i} = V'' \cap \{n''_{q}+ \alpha_{q,i} - qh: 1\leq h\leq KH_{q}\},\quad i=1,\ldots, \nu_{q},
\]and
\[
e''_{q} = \bigsqcup_{1 \leq i \leq \nu_{q}} e''_{q, i},
\]then we have
\[
\# \Big(V''\setminus \bigcup_{q\in \mathcal{R}''} e''_{q}\Big)\leq C_{0}\eta\#V''< \frac{\rho x}{10 \log x}.
\]

We recall that if $q\in \mathcal{Q}$, then $z< q \leq x/2$, and $\mathcal{R}' \cap \mathcal{R}'' = \emptyset$. By the Chinese Remainder Theorem, we take $b\equiv n'_{q}$ (mod $q$) for all $q\in \mathcal{R}'$ and $b \equiv -N - n''_{q}$ (mod $q$) for all $q\in \mathcal{R}''$. Therefore
\begin{align*}
\# (S_{x/2}(b)\cap [1,y])&\leq \# ((S'\cap [1, y])\setminus V')+
\# \Big(V'\setminus \bigcup_{q\in \mathcal{R}'}e'_{q}\Big)\\
&\leq \frac{\rho x}{10 \log x} + \frac{\rho x}{10 \log x} = \frac{\rho x}{5 \log x},
\end{align*}and
\begin{align*}
\# (S_{x/2}(-N-b)\cap [-y,-1])&\leq \# ((S''\cap [-y, -1])\setminus V'')+
\# \Big(V''\setminus \bigcup_{q\in \mathcal{R}''}e''_{q}\Big)\\
&\leq \frac{\rho x}{10 \log x} + \frac{\rho x}{10 \log x} = \frac{\rho x}{5 \log x}.
\end{align*}

Let us denote
\begin{align*}
\mathcal{A}'&:= S_{x/2}(b)\cap [1,y],\qquad &&\mathcal{A}'':= S_{x/2}(-N-b)\cap [-y,-1],\\
\mathcal{D}' &:=\{q: x/2< q\leq (3x)/4\},\qquad &&\mathcal{D}'' :=\{q: (3x)/4< q\leq x\}.
\end{align*}
 Then we have
\[
\# \mathcal{D}' > \frac{\rho x}{5\log x}\geq \#\mathcal{A}',\qquad \# \mathcal{D}'' > \frac{\rho x}{5\log x}\geq \#\mathcal{A}''.
\] Hence, we may pair up each element $a\in \mathcal{A}'$ with a unique prime $q(a)\in \mathcal{D}'$ and pair up each element $a\in \mathcal{A}''$ with a unique prime $q(a)\in \mathcal{D}''$. We take $b\equiv a - \alpha_{q(a), 1}$ (mod $q(a)$) for every $a\in \mathcal{A}'$ and $b\equiv -N - a + \alpha_{q(a), 1}$ (mod $q(a)$) for every $a\in \mathcal{A}''$ (applying again the Chinese Remainder Theorem). We obtain
\[
S_{x}(b)\cap [1,y]=\emptyset \qquad\text{and}\qquad S_{x}(-N-b)\cap [-y,-1]=\emptyset.
\]

Let $x=(\log N)/4$ and let $N$ be sufficiently large. By \eqref{SETUP:P.x.EST} and \eqref{SETUP:Def.y},  then
\[
 P(x)\leq N^{1/3}\qquad \text{and}\qquad y \asymp \log N (\log\log N)^{\delta} .
\] Let $b_{1}$ be such that $b_{1}\equiv b$ (mod $P (x)$) and $b_{1}\in [-0.3 N, -0.2 N]$. We have $S_{x}(b_{1})\cap [1,y]=\emptyset$, $S_{x}(-N -b_{1})\cap [-y,-1]=\emptyset$,  and hence
\[
S_{x}\cap (b_2 + [1,y]) = \emptyset,\qquad\text{and}\qquad S_{x}\cap (N-b_2 + [-y,-1]) = \emptyset ,
\]where $b_{2} = -b_{1}$. We set $I_{1}=\{b_{2}+1,\ldots, b_{2}+y\}$ and $I_{2}=\{N-b_{2}- y,\ldots, N-b_{2}-1\}$. Then $I_{1}$ and $I_{2}$ are strings of consecutive positive integers $n \in [0.2 N, 0.8 N]$ of length $y \gg \log N (\log\log N)^{\delta}$ and for any $n\in I_{1}\cup I_{2}$ there is a prime $B<q\leq (\log N)/4$ such that $q$ divides $f(n)$. Since
\[
\min_{0.2 N \leq n \leq 0.8 N} f(n) > \frac{\log N}{4},
\] we have $q< f(n)$, and hence $f(n)$ is composite for all $n\in I_{1}\cup I_{2}$. Finally, we take $n_{1} = b_{2} + [y/2]$ and $n_{2} = N - b_{2} - [y/2]$. This completes the proof of Theorem \ref{T1} assuming Lemma \ref{L1}. Thus we are left to prove Lemma \ref{L1}.

\section{Proof of Lemma \ref{L1}}

We deduce Lemma \ref{L1} from the following statement.

\begin{lemma}\label{L3}
The following statements hold.

\textup{(i)} One has
\begin{align}
\mathbb{E} \#(\mathbf{S}'\cap [1,y])&=\sigma y,\qquad\quad\quad\ \
\mathbb{E} \big(\#(\mathbf{S}'\cap [1,y])\big)^{2}= \bigg(1+ O\bigg(\frac{1}{\log y}\bigg)\bigg)(\sigma y)^{2},\label{S.1.OSNOVA}\\
\mathbb{E} \#(\mathbf{S}''\cap [-y,-1])&=\sigma y,\qquad\quad
\mathbb{E} \big(\#(\mathbf{S}''\cap [-y,-1])\big)^{2}= \bigg(1+ O\bigg(\frac{1}{\log y}\bigg)\bigg)(\sigma y)^{2}.\label{S.2.OSNOVA}
\end{align}

\textup{(ii)} For every $H\in \mathfrak{H}'$, every $\nu\in \mathcal{N}$ and $j\in \{1, 2\}$, we have
\begin{equation}\label{L1:lambda}
\mathbb{E}\sum_{q\in \mathcal{Q}_{H, \nu}} \Bigg(\sum_{-(K+1)y< n\leq y}\bm{\lambda}'(H; q, n)\Bigg)^{j}= \bigg(1 +
 O\bigg(\frac{\log H}{H^{M-2}}\bigg)\bigg) \big((K+2)y\big)^{j}\# \mathcal{Q}_{H,\nu}.
\end{equation}For every $H\in \mathfrak{H}''$, every $\nu\in \mathcal{N}$ and $j\in \{1, 2\}$, we have
\begin{equation}\label{L1:lambda.2}
\mathbb{E}\sum_{q\in \mathcal{Q}_{H, \nu}} \Bigg(\sum_{-y\leq n< (K+1)y}\bm{\lambda}''(H; q, n)\Bigg)^{j}= \bigg(1 +
 O\bigg(\frac{\log H}{H^{M-2}}\bigg)\bigg) \big((K+2)y\big)^{j}\# \mathcal{Q}_{H,\nu}.
\end{equation}

\textup{(iii)} For every $H\in \mathfrak{H}'$, every $\nu \in \mathcal{N}$, $i\in \{1,\ldots, \nu\}$, and $j\in \{1,2\}$, we have
\begin{align}
\mathbb{E}\sum_{n\in \mathbf{S}'\cap [1, y]} \bigg(\sum_{q\in \mathcal{Q}_{H, \nu}} \sum_{h\leq KH}\bm{\lambda}'(H; q, &n-\alpha_{q, i}-qh)\bigg)^{j}\notag\\
 &= \bigg(1 + O\bigg(\frac{\log H}{H^{M-2}}\bigg)\bigg) \bigg(\frac{\#\mathcal{Q}_{H, \nu}[KH]}{\sigma_{2}}\bigg)^{j}\sigma y.\label{L3:lamda.AP.1}
\end{align}

For every $H\in \mathfrak{H}''$, every $\nu \in \mathcal{N}$, $i\in \{1,\ldots, \nu\}$, and $j\in \{1,2\}$, we have
\begin{align}
\mathbb{E}\sum_{n\in \mathbf{S}''\cap [-y, -1]} \bigg(\sum_{q\in \mathcal{Q}_{H, \nu}} \sum_{h\leq KH}\bm{\lambda}''&(H; q, n-\alpha_{q, i}+qh)\bigg)^{j}\notag\\
 &= \bigg(1 + O\bigg(\frac{\log H}{H^{M-2}}\bigg)\bigg) \bigg(\frac{\#\mathcal{Q}_{H, \nu}[KH]}{\sigma_{2}}\bigg)^{j}\sigma y.\label{L3:lamda.AP.2}
\end{align}

\end{lemma}
\begin{proof}
The statement \eqref{S.1.OSNOVA} is part \textup{(i)} of Theorem 3 in \cite{FKMPT}, and the statements \eqref{L1:lambda} and \eqref{L3:lamda.AP.1} are parts \textup{(ii)} and \textup{(iii)} of Theorem 4 in \cite{Ford.Gabdullin}.  The statements \eqref{S.2.OSNOVA},  \eqref{L1:lambda.2} and \eqref{L3:lamda.AP.2} are proved similarly.
\end{proof}

 From part \textup{(i)} of Lemma \ref{L3} we have
\[
\mathbb{E}(\#(\mathbf{S}'\cap [1, y]) - \sigma y)^{2}\ll \frac{(\sigma y)^{2}}{\log y}\qquad\text{and}\qquad
\mathbb{E}(\#(\mathbf{S}''\cap [-y, -1]) - \sigma y)^{2}\ll \frac{(\sigma y)^{2}}{\log y}.
\]Hence by Chebyshev's inequality, we see that
\begin{align}
\mathbb{P}( \#(\mathbf{S}'\cap [1, y])\leq 2\sigma y)&= 1 - O\left(\frac{1}{\log x}\right),\label{PROB.S.EST}\\
\mathbb{P}( \#(\mathbf{S}''\cap [-y, -1])\leq 2\sigma y)&= 1 - O\left(\frac{1}{\log x}\right).\label{PROB.S.EST.2}
\end{align}

Fix $H\in \mathfrak{H}'$ and $\nu \in \mathcal{N}$. From \eqref{L1:lambda} we have
\begin{equation}\label{SUMMA.LAMBDA.1.M}
\mathbb{E}\sum_{q\in \mathcal{Q}_{H, \nu}}\Bigg(\sum_{-(K+1)y< n\leq y}\bm{\lambda}'(H;q,n)- (K+2)y\Bigg)^{2}\ll \frac{y^{2}\#\mathcal{Q}_{H,\nu}\log H}{H^{M-2}}\ll \frac{y^{2}\#\mathcal{Q}_{H,\nu}}{H^{M-2-\varepsilon}}.
\end{equation}We put
\[
\bm{\mathcal{R}}'_{H,\nu}=\Big\{q\in \mathcal{Q}_{H,\nu}: \Big|\sum_{-(K+1)y<n\leq y}\bm{\lambda}'(H;q,n)- (K+2)y\Big|\leq \frac{y}{H^{1+\varepsilon}}\Big\}.
\] We have
\begin{align}
\mathbb{E}&\# (\mathcal{Q}_{H,\nu}\setminus \bm{\mathcal{R}}'_{H,\nu})=\mathbb{E}\sum_{q\in \mathcal{Q}_{H,\nu}\setminus \bm{\mathcal{R}}'_{H,\nu}} 1\notag\\
&=\mathbb{E}\sum_{q\in \mathcal{Q}_{H,\nu}\setminus \bm{\mathcal{R}}'_{H,\nu}} \frac{\bigg(\sum_{-(K+1)y< n\leq y}\bm{\lambda}'(H;q,n)- (K+2)y\bigg)^{2}}{\bigg(\sum_{-(K+1)y< n\leq y}\bm{\lambda}'(H;q,n)- (K+2)y\bigg)^{2}}\notag\\
&\leq \frac{H^{2+2\varepsilon}}{y^{2}}\mathbb{E}\sum_{q\in \mathcal{Q}_{H,\nu}}\Bigg(\sum_{-(K+1)y< n\leq y}\bm{\lambda}'(H;q,n)- (K+2)y\Bigg)^{2}
\ll \frac{\# \mathcal{Q}_{H,\nu}}{H^{M-4-3\varepsilon}}.\label{EXP.R.H.NU.1}
\end{align}By Markov's inequality, we have
\[
\mathbb{P}\bigg(\# (\mathcal{Q}_{H,\nu}\setminus \bm{\mathcal{R}}'_{H,\nu})\leq \frac{\# \mathcal{Q}_{H,\nu}}{H^{M-4-4\varepsilon}}\bigg)=1- O(H^{-\varepsilon}).
\]We observe that for any $\beta>0$ we have (see \ref{H_range})
\begin{equation}\label{L2:H_alpha}
\sum_{H\in \mathfrak{H}} H^{-\beta}\leq \frac{1}{\big((\log x)^{\delta}\big)^{\beta}}+
\frac{1}{\big(\xi(\log x)^{\delta}\big)^{\beta}}+\frac{1}{\big(\xi^{2}(\log x)^{\delta}\big)^{\beta}}+\ldots\ll_{\beta} (\log x)^{-\delta \beta}.
\end{equation}Hence, with probability $1-O((\log x)^{-\delta\varepsilon})$ the relation
\begin{equation}\label{BASIC.I}
\# (\mathcal{Q}_{H,\nu}\setminus \bm{\mathcal{R}}'_{H,\nu})\leq \frac{\# \mathcal{Q}_{H,\nu}}{H^{M-4-4\varepsilon}}
\end{equation}holds for every $H\in\mathfrak{H}'$ and $\nu\in \mathcal{N}$ simultaneously. We put
\[
\bm{\mathcal{R}}'_{\nu}=\bigcup_{H\in \mathfrak{H}'} \bm{\mathcal{R}}'_{H,\nu}.
\]Since $H\geq (\log x)^{\delta}$ for any $H\in \mathfrak{H}$, we have for any $q\in \bm{\mathcal{R}}'_{\nu}$
\begin{equation}\label{SUM.LAMBDA.1}
\sum_{-(K+1)y< n \leq y} \bm{\lambda}' (H_{q}; q, n) = \bigg(1+O\bigg(\frac{1}{(\log x)^{\delta (1+\varepsilon)}}\bigg)\bigg)(K+2)y.
\end{equation}

Fix $H\in \mathfrak{H}''$ and $\nu \in \mathcal{N}$. From \eqref{L1:lambda.2} we have
\begin{equation}\label{SUMMA.LAMBDA.2.M}
\mathbb{E}\sum_{q\in \mathcal{Q}_{H, \nu}}\Bigg(\sum_{-y\leq n< (K+1)y}\bm{\lambda}''(H;q,n)- (K+2)y\Bigg)^{2}\ll \frac{y^{2}\#\mathcal{Q}_{H,\nu}\log H}{H^{M-2}}\ll \frac{y^{2}\#\mathcal{Q}_{H,\nu}}{H^{M-2-\varepsilon}}.
\end{equation}We put
\[
\bm{\mathcal{R}}''_{H,\nu}=\Big\{q\in \mathcal{Q}_{H,\nu}: \Big|\sum_{-y\leq n < (K+1)y}\bm{\lambda}''(H;q,n)- (K+2)y\Big|\leq \frac{y}{H^{1+\varepsilon}}\Big\}.
\] Similarly, we have
\begin{equation}\label{STRIKE.2}
\mathbb{E}\# (\mathcal{Q}_{H,\nu}\setminus \bm{\mathcal{R}}''_{H,\nu})
\ll \frac{\# \mathcal{Q}_{H,\nu}}{H^{M-4-3\varepsilon}},
\end{equation}and with probability $1-O((\log x)^{-\delta\varepsilon})$ the relation
\begin{equation}\label{BASIC.II}
\# (\mathcal{Q}_{H,\nu}\setminus \bm{\mathcal{R}}''_{H,\nu})\leq \frac{\# \mathcal{Q}_{H,\nu}}{H^{M-4-4\varepsilon}}
\end{equation}holds for every $H\in\mathfrak{H}''$ and $\nu\in \mathcal{N}$ simultaneously. We put
\[
\bm{\mathcal{R}}''_{\nu}=\bigcup_{H\in \mathfrak{H}''} \bm{\mathcal{R}}''_{H,\nu}.
\]For any $q\in \bm{\mathcal{R}}''_{\nu}$
\begin{equation}\label{SUM.LAMBDA.2}
\sum_{-y\leq n < (K+1)y} \bm{\lambda}'' (H_{q}; q, n) = \bigg(1+O\bigg(\frac{1}{(\log x)^{\delta (1+\varepsilon)}}\bigg)\bigg)(K+2)y.
\end{equation}

We work on part \textup{(iii)} of Lemma \ref{L1} using part \textup{(iii)} of Lemma \ref{L3} in a similar fashion to previous arguments. Fix $H\in \mathfrak{H}'$, $\nu\in \mathcal{N}$, and $i\in\{1,\ldots,\nu\}$. From \eqref{L3:lamda.AP.1} we have
\begin{align*}
\mathbb{E}&\sum_{n\in \mathbf{S}'\cap [1,y]} \Bigg( \sum_{q\in \mathcal{Q}_{H,\nu}}\sum_{h\leq KH}\bm{\lambda}'(H;q, n-\alpha_{q, i}-qh) - \frac{\#\mathcal{Q}_{H,\nu} [KH]}{\sigma_{2}}\Bigg)^{2}\\
&\ll \frac{\log H}{H^{M-2}} \bigg(\frac{\#\mathcal{Q}_{H,\nu}[KH]}{\sigma_2}\bigg)^{2}\sigma y
\ll \frac{1}{H^{M-2 - \varepsilon}} \bigg(\frac{\#\mathcal{Q}_{H,\nu}[KH]}{\sigma_2}\bigg)^{2}\sigma y.
\end{align*}We put
\begin{align}
\bm{\mathcal{E}}'_{H, \nu, i}=\Bigg\{&n\in \mathbf{S}'\cap [1,y]:\notag\\
 &\Bigg|\sum_{q\in \mathcal{Q}_{H,\nu}}\sum_{h\leq K H}\bm{\lambda}'(H;q, n-\alpha_{q, i}-qh) - \frac{\#\mathcal{Q}_{H, \nu} [KH]}{\sigma_{2}}\Bigg|\geq \frac{\#\mathcal{Q}_{H, \nu} [KH]}{\sigma_{2} H^{1+ \varepsilon}}\Bigg\}.\label{E.1.DEF}
\end{align}Since $M>6$ and $\varepsilon < (M-6)/7$, we have
\begin{align*}
&\mathbb{E} \#\bm{\mathcal{E}}'_{H,\nu, i}=\mathbb{E}\sum_{n\in \bm{\mathcal{E}}'_{H,\nu, i}} 1\\
&=
\mathbb{E}\sum_{n\in \bm{\mathcal{E}}'_{H,\nu,i}}\frac{\big(\sum_{q\in \mathcal{Q}_{H,\nu}}\sum_{h\leq KH}\bm{\lambda}'(H;q, n-\alpha_{q,i}-qh) - (\#\mathcal{Q}_{H,\nu} [KH])/\sigma_{2}\big)^{2}}{\big(\sum_{q\in \mathcal{Q}_{H,\nu}}\sum_{h\leq KH}\bm{\lambda}'(H;q, n-\alpha_{q,i}-qh) - (\#\mathcal{Q}_{H,\nu} [KH])/\sigma_{2}\big)^{2}}\\
&\leq \frac{\sigma_{2}^{2}H^{2 +2\varepsilon}}{(\#\mathcal{Q}_{H,\nu}[KH])^{2}}
\mathbb{E}\sum_{n\in \mathbf{S}'\cap [1,y]}\Bigg(\sum_{q\in \mathcal{Q}_{H,\nu}}\sum_{h\leq KH}\bm{\lambda}'(H;q, n-\alpha_{q,i}-qh) - \frac{\#\mathcal{Q}_{H,\nu} [KH]}{\sigma_{2}}\Bigg)^{2}\\
&\qquad\qquad\qquad\qquad\qquad\qquad\qquad\qquad\qquad\qquad\qquad\qquad\quad\quad\ \ \
\ll \frac{\sigma y}{H^{M-4-3\varepsilon}}\ll \frac{\sigma y}{H^{2}}.
\end{align*}By Markov's inequality, we have
\begin{equation}\label{MARKOV.1}
\mathbb{P}\Big(\# \bm{\mathcal{E}}'_{H,\nu,i}\leq \frac{\sigma y}{H^{1+\varepsilon}}\Big)= 1-O\Big(\frac{1}{H^{1- \varepsilon}}\Big).
\end{equation}

We next estimate the contribution from ``bad'' primes $q\in \mathcal{Q}_{H,\nu}\setminus \bm{\mathcal{R}}'_{H,\nu}$. For any $h\leq KH$, by the Cauchy-Schwarz inequality we have
\begin{align}
\mathbb{E}\sum_{n\in \mathbf{S}'\cap [1,y]} &\sum_{q\in \mathcal{Q}_{H,\nu}\setminus \bm{\mathcal{R}}'_{H,\nu}} \bm{\lambda}'(H; q, n-\alpha_{q,i}-qh)\leq \big(\mathbb{E}\# (\mathcal{Q}_{H,\nu}\setminus \bm{\mathcal{R}}'_{H,\nu})\big)^{1/2}\cdot\notag\\
&\cdot\Bigg(\mathbb{E} \sum_{q\in \mathcal{Q}_{H,\nu}\setminus \bm{\mathcal{R}}'_{H,\nu}} \bigg(\sum_{n=1}^{y} \bm{\lambda}'(H; q, n-\alpha_{q, i}-qh)\bigg)^{2}\Bigg)^{1/2}.\label{CAUSHY.LAMBDA.1}
\end{align}Given $q\in \mathcal{Q}_{H,\nu}\setminus \bm{\mathcal{R}}'_{H,\nu}$, we have
\[
\bigg|\sum_{n=1}^{y} \bm{\lambda}'(H; q, n-\alpha_{q,i}-qh)\bigg|\leq \bigg|\sum_{n=1}^{y} \bm{\lambda}'(H; q, n-\alpha_{q, i}-qh)- (K+2)y\bigg| + (K+2)y.
\] Applying \eqref{n.qh.range.1}, we obtain
\begin{align*}
\bigg|\sum_{n=1}^{y} \bm{\lambda}'(H; q, n&-\alpha_{q, i}-qh)- (K+2)y\bigg|\\
 &\leq\max\Bigg((K+2)y, \bigg|\sum_{-(K+1)y<n\leq y} \bm{\lambda}'(H; q, n)- (K+2)y\bigg|\Bigg)\\
 &\leq\bigg|\sum_{-(K+1)y<n\leq y} \bm{\lambda}'(H; q, n)- (K+2)y\bigg|+ (K+2)y.
\end{align*}

Since $(a+b)^{2}\leq 2 (a^{2}+b^{2})$, we have
\[
\bigg(\sum_{n=1}^{y} \bm{\lambda}'(H; q, n-\alpha_{q,i}-qh)\bigg)^{2}\leq 2\bigg(\sum_{-(K+1)y<n\leq y} \bm{\lambda}'(H; q, n)- (K+2)y\bigg)^{2}+ 8(K+2)^{2}y^{2}.
\]Applying \eqref{SUMMA.LAMBDA.1.M} and \eqref{EXP.R.H.NU.1}, we obtain
\begin{align*}
\mathbb{E} \sum_{q\in \mathcal{Q}_{H,\nu}\setminus \bm{\mathcal{R}}'_{H,\nu}} &\bigg(\sum_{n=1}^{y} \bm{\lambda}'(H; q, n-\alpha_{q,i}-qh)\bigg)^{2}\\
 &\leq 2\mathbb{E} \sum_{q\in \mathcal{Q}_{H,\nu}\setminus \bm{\mathcal{R}}'_{H,\nu}} \bigg(\sum_{-(K+1)y<n\leq y} \bm{\lambda}'(H; q, n)- (K+2)y\bigg)^{2}\\
 &+ 8\big((K+2)y\big)^{2}\mathbb{E}\# (\mathcal{Q}_{H,\nu}\setminus \bm{\mathcal{R}}'_{H,\nu})
 \ll \frac{y^{2}\#\mathcal{Q}_{H,\nu}}{H^{M-2 -\varepsilon}} + \frac{y^{2}\#\mathcal{Q}_{H,\nu}}{H^{M-4 - 3\varepsilon}}\ll
 \frac{y^{2}\#\mathcal{Q}_{H,\nu}}{H^{M-4-3\varepsilon}}.
\end{align*} We see from \eqref{CAUSHY.LAMBDA.1} and \eqref{EXP.R.H.NU.1} that
\[
\mathbb{E}\sum_{n\in \mathbf{S}'\cap [1,y]} \sum_{q\in \mathcal{Q}_{H,\nu}\setminus \bm{\mathcal{R}}'_{H,\nu}} \bm{\lambda}'(H; q, n-\alpha_{q,i}-qh)\ll \frac{y \#\mathcal{Q}_{H,\nu}}{H^{M-4 - 3\varepsilon}}.
\]By summing over $h\leq KH$, we get
\[
\mathbb{E}\sum_{n\in \mathbf{S}'\cap [1,y]} \sum_{q\in \mathcal{Q}_{H,\nu}\setminus \bm{\mathcal{R}}'_{H,\nu}}
 \sum_{h\leq KH}\bm{\lambda}'(H; q, n-\alpha_{q,i}-qh)\ll \frac{y \#\mathcal{Q}_{H,\nu}}{H^{M-5 - 3\varepsilon}}.
\]We put
\begin{equation}\label{T.1.DEF}
\bm{\mathcal{T}}'_{H,\nu,i}=\bigg\{n\in \mathbf{S}'\cap [1, y]:
\sum_{q\in \mathcal{Q}_{H,\nu}\setminus \bm{\mathcal{R}}'_{H,\nu}}
 \sum_{h\leq KH}\bm{\lambda}'(H; q, n-\alpha_{q,i}-qh)\geq \frac{ \#\mathcal{Q}_{H,\nu}[KH]}{H^{1+\varepsilon}\sigma_2}\bigg\}.
\end{equation}We have
\begin{align*}
\mathbb{E} \#\bm{\mathcal{T}}'_{H,\nu,i}&=
\mathbb{E}\sum_{n\in \bm{\mathcal{T}}'_{H,\nu,i}} \frac{\sum_{q\in \mathcal{Q}_{H,\nu}\setminus \bm{\mathcal{R}}'_{H,\nu}}
 \sum_{h\leq KH}\bm{\lambda}'(H; q, n-\alpha_{q,i}-qh)}{\sum_{q\in \mathcal{Q}_{H,\nu}\setminus \bm{\mathcal{R}}'_{H,\nu}}
 \sum_{h\leq KH}\bm{\lambda}'(H; q, n-\alpha_{q,i}-qh)}\\
 &\leq \frac{H^{1+\varepsilon}\sigma_2}{\#\mathcal{Q}_{H,\nu}[KH]}
 \mathbb{E}\sum_{n\in \mathbf{S}'\cap [1,y]}\sum_{q\in \mathcal{Q}_{H,\nu}\setminus \bm{\mathcal{R}}'_{H,\nu}}
 \sum_{h\leq KH}\bm{\lambda}'(H; q, n-\alpha_{q,i}-qh)\\
 &\ll \frac{y\sigma_2}{H^{M-5-4\varepsilon}} \ll \sigma y \frac{\log H}{H^{M-5 - 4\varepsilon}}\ll \frac{\sigma y}{H^{M-5-5\varepsilon}}.
\end{align*}By Markov's inequality, we have
\begin{equation}\label{MARKOV.2}
\mathbb{P}\bigg(\#\bm{\mathcal{T}}'_{H,\nu,i} \leq \frac{\sigma y}{H^{1+\varepsilon}}\bigg)=1-O\Big(\frac{1}{H^{M-6-6\varepsilon}}\Big).
\end{equation} We see from \eqref{L2:H_alpha}, \eqref{MARKOV.1} and \eqref{MARKOV.2} that with probability $1 - O\big((\log x)^{-\delta \gamma}\big)$, where
\[
\gamma = \min (1-\varepsilon, M-6-6\varepsilon) >0,
\] the relations
\[
\#\bm{\mathcal{E}}'_{H,\nu,i} \leq \frac{\sigma y}{H^{1+\varepsilon}},\qquad
\#\bm{\mathcal{T}}'_{H,\nu,i} \leq \frac{\sigma y}{H^{1+\varepsilon}}
\]hold for every $H\in \mathfrak{H}'$, $\nu\in \mathcal{N}$, and $i\in\{1,\ldots,\nu\}$ simultaneously.

Fix $H\in \mathfrak{H}''$, $\nu\in \mathcal{N}$, and $i\in\{1,\ldots,\nu\}$. From \eqref{L3:lamda.AP.2} we have
\begin{align*}
\mathbb{E}&\sum_{n\in \mathbf{S}''\cap [-y,-1]} \Bigg( \sum_{q\in \mathcal{Q}_{H,\nu}}\sum_{h\leq KH}\bm{\lambda}''(H;q, n-\alpha_{q, i}+qh) - \frac{\#\mathcal{Q}_{H,\nu} [KH]}{\sigma_{2}}\Bigg)^{2}\\
&\ll \frac{\log H}{H^{M-2}} \bigg(\frac{\#\mathcal{Q}_{H,\nu}[KH]}{\sigma_2}\bigg)^{2}\sigma y
\ll \frac{1}{H^{M-2 - \varepsilon}} \bigg(\frac{\#\mathcal{Q}_{H,\nu}[KH]}{\sigma_2}\bigg)^{2}\sigma y.
\end{align*}We put
\begin{align}
\bm{\mathcal{E}}''_{H, \nu, i}=\Bigg\{&n\in \mathbf{S}''\cap [-y,-1]:\notag\\
 &\Bigg|\sum_{q\in \mathcal{Q}_{H,\nu}}\sum_{h\leq K H}\bm{\lambda}''(H;q, n-\alpha_{q, i}+qh) - \frac{\#\mathcal{Q}_{H, \nu} [KH]}{\sigma_{2}}\Bigg|\geq \frac{\#\mathcal{Q}_{H, \nu} [KH]}{\sigma_{2} H^{1+ \varepsilon}}\Bigg\}.\label{E.2.DEF}
\end{align}Arguing as above, we obtain
\[
\mathbb{E} \#\bm{\mathcal{E}}''_{H,\nu, i}
\ll \frac{\sigma y}{H^{M-4-3\varepsilon}}\ll \frac{\sigma y}{H^{2}}.
\]By Markov's inequality, we have
\begin{equation}\label{MARKOV.3}
\mathbb{P}\Big(\# \bm{\mathcal{E}}''_{H,\nu,i}\leq \frac{\sigma y}{H^{1+\varepsilon}}\Big)= 1-O\Big(\frac{1}{H^{1- \varepsilon}}\Big).
\end{equation}

We next estimate the contribution from ``bad'' primes $q\in \mathcal{Q}_{H,\nu}\setminus \bm{\mathcal{R}}''_{H,\nu}$. For any $h\leq KH$, by the Cauchy-Schwarz inequality we have
\begin{align}
\mathbb{E}\sum_{n\in \mathbf{S}''\cap [-y,-1]} &\sum_{q\in \mathcal{Q}_{H,\nu}\setminus \bm{\mathcal{R}}''_{H,\nu}} \bm{\lambda}''(H; q, n-\alpha_{q,i}+qh)\leq \big(\mathbb{E}\# (\mathcal{Q}_{H,\nu}\setminus \bm{\mathcal{R}}''_{H,\nu})\big)^{1/2}\cdot\notag\\
&\cdot\Bigg(\mathbb{E} \sum_{q\in \mathcal{Q}_{H,\nu}\setminus \bm{\mathcal{R}}''_{H,\nu}} \bigg(\sum_{-y\leq n\leq -1} \bm{\lambda}'(H; q, n-\alpha_{q, i}+qh)\bigg)^{2}\Bigg)^{1/2}.\label{CAUSHY.LAMBDA.2}
\end{align}Given $q\in \mathcal{Q}_{H,\nu}\setminus \bm{\mathcal{R}}''_{H,\nu}$, we have
\[
\bigg|\sum_{-y\leq n\leq -1} \bm{\lambda}''(H; q, n-\alpha_{q,i}+qh)\bigg|\leq \bigg|\sum_{-y\leq n\leq -1} \bm{\lambda}''(H; q, n-\alpha_{q, i}+qh)- (K+2)y\bigg| + (K+2)y.
\] Applying \eqref{n.qh.range.2}, we obtain
\begin{align*}
\bigg|\sum_{-y\leq n\leq -1} \bm{\lambda}''(H; q, n&-\alpha_{q, i}+qh)- (K+2)y\bigg|\\
 &\leq\max\Bigg((K+2)y, \bigg|\sum_{-y \leq n< (K+1)y} \bm{\lambda}''(H; q, n)- (K+2)y\bigg|\Bigg)\\
 &\leq\bigg|\sum_{-y \leq n < (K+1)y} \bm{\lambda}''(H; q, n)- (K+2)y\bigg|+ (K+2)y.
\end{align*}

Since $(a+b)^{2}\leq 2 (a^{2}+b^{2})$, we get
\[
\bigg(\sum_{-y\leq n\leq -1} \bm{\lambda}''(H; q, n-\alpha_{q,i}+qh)\bigg)^{2}\leq 2\bigg(\sum_{-y \leq n< (K+1)y} \bm{\lambda}''(H; q, n)- (K+2)y\bigg)^{2}+ 8(K+2)^{2}y^{2}.
\]Applying \eqref{SUMMA.LAMBDA.2.M} and \eqref{STRIKE.2}, we obtain
\begin{align*}
\mathbb{E} \sum_{q\in \mathcal{Q}_{H,\nu}\setminus \bm{\mathcal{R}}''_{H,\nu}} &\bigg(\sum_{-y\leq n\leq -1} \bm{\lambda}''(H; q, n-\alpha_{q,i}+qh)\bigg)^{2}\\
 &\leq 2\mathbb{E} \sum_{q\in \mathcal{Q}_{H,\nu}\setminus \bm{\mathcal{R}}''_{H,\nu}} \bigg(\sum_{-y\leq n< (K+1) y} \bm{\lambda}''(H; q, n)- (K+2)y\bigg)^{2}\\
 &+ 8\big((K+2)y\big)^{2}\mathbb{E}\# (\mathcal{Q}_{H,\nu}\setminus \bm{\mathcal{R}}''_{H,\nu})
 \ll \frac{y^{2}\#\mathcal{Q}_{H,\nu}}{H^{M-2 -\varepsilon}} + \frac{y^{2}\#\mathcal{Q}_{H,\nu}}{H^{M-4 - 3\varepsilon}}\ll
 \frac{y^{2}\#\mathcal{Q}_{H,\nu}}{H^{M-4-3\varepsilon}}.
\end{align*} By \eqref{CAUSHY.LAMBDA.2} and \eqref{STRIKE.2}, we have
\[
\mathbb{E}\sum_{n\in \mathbf{S}''\cap [-y,-1]} \sum_{q\in \mathcal{Q}_{H,\nu}\setminus \bm{\mathcal{R}}''_{H,\nu}} \bm{\lambda}''(H; q, n-\alpha_{q,i}+qh)\ll \frac{y \#\mathcal{Q}_{H,\nu}}{H^{M-4 - 3\varepsilon}}.
\]By summing over $h\leq KH$, we get
\[
\mathbb{E}\sum_{n\in \mathbf{S}''\cap [-y,-1]} \sum_{q\in \mathcal{Q}_{H,\nu}\setminus \bm{\mathcal{R}}''_{H,\nu}}
 \sum_{h\leq KH}\bm{\lambda}''(H; q, n-\alpha_{q,i}+qh)\ll \frac{y \#\mathcal{Q}_{H,\nu}}{H^{M-5 - 3\varepsilon}}.
\]We put
\begin{equation}\label{T.2.DEF}
\bm{\mathcal{T}}''_{H,\nu,i}=\bigg\{n\in \mathbf{S}''\cap [-y, -1]:
\sum_{q\in \mathcal{Q}_{H,\nu}\setminus \bm{\mathcal{R}}''_{H,\nu}}
 \sum_{h\leq KH}\bm{\lambda}''(H; q, n-\alpha_{q,i}+qh)\geq \frac{ \#\mathcal{Q}_{H,\nu}[KH]}{H^{1+\varepsilon}\sigma_2}\bigg\}.
\end{equation}We see that
\[
\mathbb{E} \#\bm{\mathcal{T}}''_{H,\nu,i}\ll \frac{y\sigma_2}{H^{M-5-4\varepsilon}} \ll \sigma y \frac{\log H}{H^{M-5 - 4\varepsilon}}\ll \frac{\sigma y}{H^{M-5-5\varepsilon}}.
\]By Markov's inequality, we have
\begin{equation}\label{MARKOV.4}
\mathbb{P}\bigg(\#\bm{\mathcal{T}}''_{H,\nu,i} \leq \frac{\sigma y}{H^{1+\varepsilon}}\bigg)=1-O\Big(\frac{1}{H^{M-6-6\varepsilon}}\Big).
\end{equation} We see from \eqref{L2:H_alpha}, \eqref{MARKOV.3} and \eqref{MARKOV.4} that with probability $1 - O\big((\log x)^{-\delta \gamma}\big)$, where
\[
\gamma = \min (1-\varepsilon, M-6-6\varepsilon) >0,
\] the relations
\[
\#\bm{\mathcal{E}}''_{H,\nu,i} \leq \frac{\sigma y}{H^{1+\varepsilon}},\qquad
\#\bm{\mathcal{T}}''_{H,\nu,i} \leq \frac{\sigma y}{H^{1+\varepsilon}}
\]hold for every $H\in \mathfrak{H}''$, $\nu\in \mathcal{N}$, and $i\in\{1,\ldots,\nu\}$ simultaneously.

Now we make a choice of $b$ (mod $P(z)$). We consider the event that $\# (\mathbf{S}' \cap [1,y])\leq 2\sigma y$ and that $\# (\mathbf{S}'' \cap [-y,-1])\leq 2\sigma y$, that for each $H\in \mathfrak{H}'$, $\nu \in \mathcal{N}$, $i\leq \nu$, the sets $\bm{\mathcal{E}}'_{H,\nu,i}$, $\bm{\mathcal{T}}'_{H,\nu,i}$ have size at most $(\sigma y) H^{-1-\varepsilon}$, that for each $H\in \mathfrak{H}''$, $\nu \in \mathcal{N}$, $i\leq \nu$, the sets $\bm{\mathcal{E}}''_{H,\nu,i}$, $\bm{\mathcal{T}}''_{H,\nu,i}$ have size at most $(\sigma y) H^{-1-\varepsilon}$, that for each $H\in\mathfrak{H}'$ and $\nu\in\mathcal{N}$ the relation \eqref{BASIC.I} holds, and that for each $H\in\mathfrak{H}''$ and $\nu\in \mathcal{N}$ the relation \eqref{BASIC.II} holds. By the above discussion, this event holds with probability $1- o(1)$ as $x\to \infty$, and so this probability is at least $1/2$ provided that $x$ is large enough in terms of $\delta$, $M$, $K$, and $\varepsilon$. From now, we fix a $b$ mod $P(z)$ such that it is so, and thus all of our random sets and weights become deterministic. We see that, for each $\nu\in \mathcal{N}$, $\mathcal{R}'_{\nu}$ and $\mathcal{R}''_{\nu}$ are non-empty subsets of $\mathcal{Q}'_{\nu}$ and $\mathcal{Q}''_{\nu}$ respectively, and thus, by \eqref{SUM.LAMBDA.1} and \eqref{SUM.LAMBDA.2}, parts \textup{(i)} and \textup{(ii)} of Lemma \ref{L1} are verified. Now we verify parts \textup{(iii)} and \textup{(iv)} of Lemma \ref{L1}.

Fix $\nu\in \mathcal{N}$ and $i\in\{1,\ldots, \nu\}$. We set
\[
V'_{\nu, i} = (S'\cap [1,y]) \setminus \bigcup_{H\in \mathfrak{H}'}
(\mathcal{E}'_{H,\nu,i}\cup \mathcal{T}'_{H,\nu,i}).
\]By \eqref{L2:H_alpha}, we have
\begin{align*}
\#\Bigg(\bigcup_{H\in \mathfrak{H}'}
(\mathcal{E}'_{H,\nu,i}\cup \mathcal{T}'_{H,\nu,i})\Bigg)\ll \sigma y
\sum_{H\in \mathfrak{H}'}\frac{1}{H^{1+\varepsilon}}&\ll \frac{\sigma y}{(\log x)^{\delta (1+\varepsilon)}}\\
&\ll \frac{x}{(\log x)^{1+\delta \varepsilon}}< \frac{\rho x}{10 B^{2} \log x},
\end{align*}verifying \eqref{V.1.DOPOLNENIYE}.

We fix arbitrary $n\in V'_{\nu, i}$. For such $n$, the inequalities  \eqref{E.1.DEF} and \eqref{T.1.DEF} both fail, and therefore for each $H\in \mathfrak{H}'$,
\begin{align*}
\sum_{q\in \mathcal{R}'_{H,\nu}}\sum_{h\leq K H}\lambda '(H; q, n-\alpha_{q,i}-qh)&=
\bigg(1+ O\bigg(\frac{1}{H^{1+\varepsilon}}\bigg)\bigg) \frac{\#\mathcal{Q}_{H,\nu}[KH]}{\sigma_2}\\
&=\bigg(1+ O\bigg(\frac{1}{(\log x)^{\delta(1+\varepsilon)}}\bigg)\bigg) \frac{\#\mathcal{Q}_{H,\nu}[KH]}{\sigma_2}.
\end{align*}Summing over all $H\in \mathfrak{H}'$, we have
\begin{align*}
\sum_{q\in \mathcal{R}'_{\nu}}\sum_{h\leq KH_{q}}\lambda '(H; q, n-\alpha_{q, i}-qh)&=
\sum_{H\in \mathfrak{H}'}\sum_{q\in \mathcal{R}'_{H,\nu}}\sum_{h\leq KH_{q}}\lambda'(H; q, n-\alpha_{q, i}-qh)\\
&=\bigg(1+O\bigg(\frac{1}{(\log x)^{\delta (1+\varepsilon)}}\bigg)\bigg)C'_{\nu} (K+2)y,
\end{align*}where
\[
C'_{\nu}= \frac{1}{(K+2)y}\sum_{H\in \mathfrak{H}'}\frac{\# \mathcal{Q}_{H,\nu} [KH]}{\sigma_{2}}.
\] Note that $C'_{\nu}$ depends on $x$, $K$, $M$, $\xi$, $\delta$, and $\nu$, but not on $n$ and $i$. Since
\[
[KH]=KH\Big(1+ O\Big(\frac{1}{H}\Big)\Big) = KH\Big(1+ O\Big(\frac{1}{(\log x)^{\delta}}\Big)\Big),
\] we get, using \eqref{QH.NU.ASYMPT} and \eqref{SIGMA.2},
\[
C'_{\nu}\sim \frac{\rho_{\nu}K (1-1/\xi)}{(K+2)M} \sum_{H\in\mathfrak{H}'}\frac{1}{\log H}.
\]Recalling the definition of $\mathfrak{H}'$, we see that
\[
C'_{\nu}\sim \frac{\rho_{\nu}K (1-1/\xi)}{2(K+2)M \log \xi} \sum_{j} \frac{1}{j},
\]where $j$ runs over the interval
\[
\frac{\delta \log\log x}{2 \log \xi}(1+o(1)) \leq j \leq \frac{\log\log x}{4\log \xi} (1+o(1)).
\]We thus obtain
\[
C'_{\nu}\sim \frac{\rho_{\nu} K}{2(K+2) M} \frac{1-1/\xi}{\log \xi} \log \Big(\frac{1}{2\delta}\Big).
\]

Similarly, for fixed $\nu\in \mathcal{N}$ and $i\in\{1,\ldots, \nu\}$, we set
\[
V''_{\nu, i} = (S''\cap [-y,-1]) \setminus \bigcup_{H\in \mathfrak{H}''}
(\mathcal{E}''_{H,\nu,i}\cup \mathcal{T}''_{H,\nu,i}).
\]By \eqref{L2:H_alpha}, we have
\begin{align*}
\#\Bigg(\bigcup_{H\in \mathfrak{H}''}
(\mathcal{E}''_{H,\nu,i}\cup \mathcal{T}''_{H,\nu,i})\Bigg)\ll \sigma y
\sum_{H\in \mathfrak{H}''}\frac{1}{H^{1+\varepsilon}}&\ll \frac{\sigma y}{(\log x)^{\delta (1+\varepsilon)}}\\
&\ll \frac{x}{(\log x)^{1+\delta \varepsilon}}< \frac{\rho x}{10 B^{2} \log x},
\end{align*}verifying \eqref{V.2.DOPOLNENIYE}.

We fix arbitrary $n\in V''_{\nu, i}$. For such $n$, the inequalities  \eqref{E.2.DEF} and \eqref{T.2.DEF} both fail, and therefore for each $H\in \mathfrak{H}''$,
\begin{align*}
\sum_{q\in \mathcal{R}''_{H,\nu}}\sum_{h\leq K H}\lambda ''(H; q, n-\alpha_{q,i}+qh)&=
\bigg(1+ O\bigg(\frac{1}{H^{1+\varepsilon}}\bigg)\bigg) \frac{\#\mathcal{Q}_{H,\nu}[KH]}{\sigma_2}\\
&=\bigg(1+ O\bigg(\frac{1}{(\log x)^{\delta(1+\varepsilon)}}\bigg)\bigg) \frac{\#\mathcal{Q}_{H,\nu}[KH]}{\sigma_2}.
\end{align*}Summing over all $H\in \mathfrak{H}''$, we have
\begin{align*}
\sum_{q\in \mathcal{R}''_{\nu}}\sum_{h\leq KH_{q}}\lambda ''(H; q, n-\alpha_{q, i}+qh)&=
\sum_{H\in \mathfrak{H}''}\sum_{q\in \mathcal{R}''_{H,\nu}}\sum_{h\leq KH_{q}}\lambda''(H; q, n-\alpha_{q, i}+qh)\\
&=\bigg(1+O\bigg(\frac{1}{(\log x)^{\delta (1+\varepsilon)}}\bigg)\bigg)C''_{\nu} (K+2)y,
\end{align*}where
\[
C''_{\nu}= \frac{1}{(K+2)y}\sum_{H\in \mathfrak{H}''}\frac{\# \mathcal{Q}_{H,\nu} [KH]}{\sigma_{2}}.
\] Note that $C''_{\nu}$ depends on $x$, $K$, $M$, $\xi$, $\delta$, and $\nu$, but not on $n$ and $i$. Arguing as above, we obtain
\[
C''_{\nu}\sim \frac{\rho_{\nu} K}{2(K+2) M} \frac{1-1/\xi}{\log \xi} \log \Big(\frac{1}{2\delta}\Big).
\]

Since $\delta < C(1/2)$, we see from \eqref{C.RHO.INEQUALITY} that
\[
C'_{\nu}\geq 10^{2\delta}\rho_{\nu}  \quad\text{and}\quad C''_{\nu} \geq 10^{2\delta}\rho_{\nu}
 \] provided that $(M-6)$ and $(\xi - 1)$ are sufficiently small in terms of $\delta$, $K$ is sufficiently large in terms of $\delta$, $0 < \varepsilon < (M-6)/7$, and $x$ is sufficiently large in terms of $\delta$, $M$, $\xi$, $K$, $\varepsilon$. Also $C'_{\nu} \leq 100 \rho_{\nu}$ and $C''_{\nu} \leq 100 \rho_{\nu}$  due to $\delta > 10^{-6}$.

So parts \textup{(iii)} and \textup{(iv)} of Lemma \ref{L1} are verified. Lemma \ref{L1} is proved.

\section{Acknowledgements}

This work is an output of a research project (HSE-BR-2025-024) implemented as part of the Basic Research Program at HSE University.

\end{document}